\documentclass[final,3p]{elsarticle}

\usepackage[normalem]{ulem}
\usepackage{tikz}
\usepackage{epstopdf}

\usepackage{graphicx}
\usepackage{color}
\usepackage{amsthm}
\usepackage{graphics}
\usepackage{amsmath,amssymb,amsfonts}
\usepackage{epsfig}
\usepackage{subfigure}
\usepackage{multirow}

\biboptions{sort&compress}

\newtheorem{theorem}{Theorem}[section]

\newtheorem{proposition}[theorem]{Proposition}

\newtheorem{definition}[theorem]{Definition}
\newtheorem{example}[theorem]{Example}

\newtheorem{remark}[theorem]{Remark}

\numberwithin{equation}{section}

\newcommand\cM{{\mathcal M}}

\def\NN{\mathbb  N}

\newcommand\RR{\mathbb{{R}}}
\newcommand\CC{\mathbb{{C}}}
\newcommand\ZZ{\mathbb{{Z}}}

\newcommand\ra{{\rm a}}

\newcommand\ba{{\bf a}}

\newcommand\bc{{\bf c}}

\newcommand\bomega{{\boldsymbol\omega}}

\newcommand\cA{{\mathcal A}}

\newcommand{\mfrac}[2]%
{\raisebox{0.5pt}{\footnotesize$\dfrac{#1}{#2}$}}
\newcommand{\mbinom}[2]%
{\raisebox{0.5pt}{\footnotesize$\dbinom{#1}{#2}$}}
\def\smmat\{#1&#2\cr#3&#4\}%
   {\!\!\pmatrix{\!#1\!&\!#2\!\cr\!#3\!&\!#4\!}\!\!}
\newcommand\scrm{{\raise0.5pt\hbox{-}}}
\def\eop{{ \vrule height7pt width7pt depth0pt}\par\bigskip} 
\newcommand\ie{{\it\thinspace i.e.}}
{\par\noindent\textbf{Proof:}~}

\begin{document}

\title{Limits of level and parameter dependent subdivision schemes:\\ a matrix approach}

\author[mc]{Maria Charina\corref{cor1}}
\ead{maria.charina@univie.ac.at}

\author[cc]{Costanza Conti}
\ead{costanza.conti@unifi.it}

\author[ng]{Nicola Guglielmi}
\ead{guglielm@units.it}

\author[vp]{Vladimir Protasov}
\ead{v-protassov@yandex.ru}

\cortext[cor1]{Corresponding author}
\address[mc]{Fakult\"at Mathematik, Universit\"at Wien, Austria}
\address[cc]{DIEF, University of Firenze, Italy}
\address[ng]{INDIRIZZO, University of L'Aquila, Italy}
\address[vp]{Moscow State University, Russia}
\begin{abstract}
In this paper, we present a new matrix approach for the analysis
of subdivision schemes whose non-stationarity is due to linear
dependency on parameters whose values vary in a compact set.
Indeed, we show how to check the convergence in $C^{\ell}(\RR^s)$
and determine the H\"older regularity of such level and parameter
dependent schemes efficiently via the joint spectral radius
approach.
The efficiency of this method and the important role of the
parameter dependency are demonstrated on several examples of
subdivision schemes whose properties improve the properties of the
corresponding stationary schemes. Moreover, we derive necessary
criteria for a function to be generated by some level dependent
scheme and, thus, expose the limitations of such schemes.
\end{abstract}


\begin{keyword} Level dependent (non-stationary) subdivision schemes, tension parameter, sum rules,
H\"{o}lder regularity, joint spectral radius.
\end{keyword}
\maketitle


\section{Introduction}\label{sec:intro}
We analyze convergence and H\"older regularity of multivariate
level dependent (non-stationary) subdivision schemes whose masks
depend linearly on one or several parameters. For this type of
schemes, which include well-known schemes with tension parameters
\cite{BCR2007_1, BCR2007_2, CoGoPi07, ContiRomani10, FMW2007, FMW2010}, the
theoretical results from \cite{US2014} are applicable, but not
always efficient. Indeed, if the level dependent parameters vary
in some compact set, then the set of the so-called limit points
(see \cite{US2014}) of the corresponding sequence of
non-stationary masks exists, but cannot be determined explicitly.
This hinders the regularity analysis of such schemes. Thus, we
present a different perspective on the results in \cite{US2014}
and derive a new general method for convergence and regularity
analysis of such level and parameter dependent schemes. The
practical efficiency of this new method is illustrated on several
examples. We also derive necessary criteria that allow us to
describe the class of functions that can be generated by
non-stationary subdivision schemes. Indeed, we show how to
characterize such functions by the special property of the zeros
of their Fourier transforms.

Subdivision schemes are iterative algorithms for generating curves
and surfaces from given control points of a mesh. They are easy to
implement and intuitive in use. These and other nice mathematical
properties of subdivision schemes motivate their popularity in
applications, i.e. in modelling of freeform curves and surfaces,
approximation and interpolation of functions, computer animation,
signal and image processing etc. Non-stationary subdivision
schemes extend the variety of different shapes generated by
stationary subdivision. Indeed, the level dependency enables to
generate new classes of functions such as exponential polynomials,
exponential B-splines, etc. This gives a new impulse to
development of subdivision schemes and enlarges the scope of their
applications, e.g. in biological imaging \cite{DelUnser2012,
Noi2014}, geometric design \cite{ReifPeter08,WW02} or isogeometric
analysis \cite{Schroder2000, Umlauf10}.

The main challenges in the analysis of any subdivision scheme are
its convergence (in various function spaces), the regularity of
its limit functions and its generation and reproduction
properties. The important role of the matrix approach for
regularity analysis of stationary subdivision schemes is
well-known. It allows to reduces the analysis to the computation
or estimation of the joint spectral radius of the finite set of
square matrices derived from the subdivision mask. Recent advances
in the joint spectral radius computation~\cite{GP13, MR14} makes
the matrix approach very precise and efficient. In the
non-stationary setting, however, this approach has never been
applied because of the several natural obstacles. First of all,
the matrix products that emerge in the study of non-stationary
schemes have a different form than those usually analyzed by the
joint spectral radius techniques. Secondly, the masks of
non-stationary schemes do not necessarily satisfy sum rules, which
destroys the relation between the convergence of the scheme and
spectral properties of its transition matrices. All those
difficulties were put aside by the results in~\cite{US2014}, where
the matrix approach was extended to general non-stationary
setting.

In this paper, in Section \ref{sec:parameters}, we make the next
step and consider level and parameter dependent subdivision
schemes whose masks include tension parameters, used to control
the properties of the subdivision limit. Mostly, the tension
parameters are level dependent and influence the asymptotic
behavior of the scheme. If this is the case, the scheme can be
analyzed by \cite[Theorem 2]{US2014}, which states that the
convergence and H\"older regularity of any such non-stationary
scheme depends on the joint spectral radius of the matrices
generated by the so-called limit points of the sequence of
level-dependent masks. In Theorem~3.5, we show that for the
schemes with linear dependence on these parameters, the result of
\cite[Theorem 2]{US2014} can be simplified and be made more
practical, see examples in Section~\ref{subsec:examples}. In
Section \ref{sec:limitations}, we address the problem of
reproduction property of subdivision schemes and of characterizing
the functions that can be generated by non-stationary subdivision
schemes. This question is crucial in many aspects. For instance,
the reproduction of exponential polynomials is strictly connected
to the approximation order of a subdivision scheme and to its
regularity \cite{ContiRomaniYoon2015}. Essentially, the higher is
the number of exponential polynomials that are being reproduced,
the higher is the approximation order and the possible regularity
of the corresponding scheme.

\section{Background}
\bigskip \noindent Let
$M=mI \in \ZZ^{s \times s}$, $|m| \ge 2$, be a dilation matrix and
$E=\{0, \ldots,|m|-1\}^s$ be the set of the coset representatives
of $\ZZ^s / M \ZZ^s$. We study subdivision schemes given by the
sequence $\{S_{\ba^{(r)}}, \ r \in \NN\}$ of subdivision operators
$S_{\ba^{(r)}}: \ell(\ZZ^s) \rightarrow \ell(\ZZ^s)$ that define
the subdivision rules by
$$
 (S_{\ba^{(r)}} \bc)(\alpha)=\sum_{\beta \in \ZZ^s} \ra_{\alpha-M\beta}^{(r)} c(\beta),
 \quad \alpha \in \ZZ^s.
$$
The masks ${\mathbf a}^{(r)}=\{ \ra_{\alpha}^{(r)}, \ \alpha \in
\ZZ^s\}$, $r \in \NN$, are sequences of real numbers
$\ra_{\alpha}^{(r)}$ and are assumed to be all supported in $\{0,
\ldots,N\}^s$, $N \in \NN$. For the given set
\begin{equation} \label{def:K}
  K= \sum_{r=1}^\infty M^{-1}G,\quad G= \{-|m|,\ldots ,N + 1\}^s,
\end{equation}
the masks define the square matrices
\begin{equation}\label{del:matrices}
 A^{(r)}_{\varepsilon}=\left( \ra^{(r)}_{M\alpha+\varepsilon-\beta}\right)_{\alpha,\beta \in K},
 \quad r \in \NN, \quad \varepsilon \in E.
\end{equation}
We assume that the level dependent symbols
$$
 a^{(r)}(z)= \sum_{\alpha \in \ZZ^s} \ra_{\alpha}^{(r)} z^{\alpha}, \quad
 z^{\alpha}=z_1^{\alpha_1} \cdot \ldots \cdot z_s^{\alpha_s},\quad z\in
 \left(\CC \setminus \{0\} \right)^s.
$$
of the subdivision scheme
$$
 c^{(r+1)}=S_{\ba^{(r)}} c^{(r)}=S_{\ba^{(r)}}S_{\ba^{(r-1)}} \ldots S_{\ba^{(1)}} c^{(1)},
 \quad r \in \NN,
$$
satisfy sum rules of order $\ell+1$, $\ell \in \NN_0$. For more
details on sum rules see e.g \cite{Cabrelli,
CaravettaDahmenMicchelli, JetterPlonka, JiaJiang}.

\begin{definition}  Let  $\ell \in \NN_0$, $r \in \NN$.
The symbol $a^{(r)}(z)$, $z \in (\CC \setminus \{0\})^s$, satisfies sum
rules of order $\ell+1$ if
\begin{equation} \label{def:sumrules}
 a^{(r)}(1, \ldots,1)=|m|^s \quad\hbox{and}\quad
 \max_{|\eta| \le \ell}\  \max_{\epsilon \in \Xi \setminus
\{1\}} | D^\eta a^{(r)}(\epsilon)|=0\,,
\end{equation}
where $
 \Xi=\{e^{-i\frac{2\pi}{|m|}\varepsilon}=(e^{-i\frac{2\pi}{|m|}\varepsilon_1},
 \ldots,
 e^{-i\frac{2\pi}{|m|}\varepsilon_s}), \ \varepsilon \in E\}$ and
 $D^\eta=\frac{\partial^{\eta_1}}{\partial z_1^{\eta_1}} \ldots \frac{\partial^{\eta_s}}{\partial
 z_s^{\eta_s}}$.
\end{definition}

The assumption that all symbols $a^{(r)}(z)$ satisfy sum rules of
order $\ell+1$, guarantees that the matrices
$A^{(r)}_{\varepsilon}$, $\varepsilon \in E$, $r \in \NN$, in
\eqref{del:matrices} have common left-eigenvectors of the form
$$
  \left( p(\alpha) \right)_{\alpha \in K}, \quad p \in \Pi_\ell,
$$
where $\Pi_\ell$ is the space of polynomials of degree less than
or equal to $\ell$. Thus, the matrices $A^{(r)}_{\varepsilon}$,
$\varepsilon \in E$, $r \in \NN$,  possess a common linear
subspace $V_\ell \subset \RR^{|K|}$ orthogonal to the span of the
common left-eigenvectors of $A^{(r)}_{\varepsilon}$, $\varepsilon \in
E$, $r \in \NN$. The spectral properties of the set
$$
 {\cal T}=\{ A^{(r)}_{\varepsilon}|_{V_\ell},  \ \varepsilon \in E, \ r \in \NN \}
$$
determine the regularity of the non-stationary scheme,
see \cite{US2014}.

\begin{remark}
In the  univariate case, i.e. $|m|=|M|$, the assumption that the
symbols $a^{(r)}(z)$, $r \in \NN$, satisfy sum rules of order
$\ell+1$ implies that
$$
 a^{(r)}(z)=(1+z+ \ldots + z^{|m|-1})^\ell \sum_{\alpha \in \ZZ} b^{(r)}_{\alpha} z^{\alpha},
 \quad z \in \CC \setminus \{0\},
$$
and
\begin{equation}\label{eq:defA_eps_rest_V_uni}
    A^{(r)}_{\varepsilon}|_{V_\ell}=\left( b^{(r)}_{M\alpha+\varepsilon-\beta}\right)_
    {\alpha, \beta \in \{0, \ldots, N-\ell\}}, \quad \varepsilon \in E.
\end{equation}
In the multivariate case, the explicit form of the matrices
$A^{(r)}_{\varepsilon}|_{V_\ell}$, $\varepsilon \in E$, $r \in
\NN$, depends on the choice of the basis of ${V_\ell}$, see e.g.
\cite[Section 3.1]{US2014} or \cite{Cabrelli}.
\end{remark}

\begin{definition} \label{def:Cellconvergence}
A subdivision scheme $\{S_{\ba^{(r)}}, \  r\in \NN \}$ is
\emph{$C^\ell$-convergent}, if for any initial sequence $\bc \in
\ell_\infty(\ZZ^s)$ there exists the
 limit function $g_\bc \in C^\ell(\RR^s)$ such that for any test function
$f \in C^\ell(\RR^s)$
\begin{equation}\label{eq:C^l_convergence}
  \lim_{k \to\infty} \Big \|  g_{\bc}(\cdot) - \sum_{\alpha \in \ZZ^s} S_{\ba^{(r)}}
 S_{\ba^{(r-1)}}  \ldots  S_{\ba^{(1)}}c(\alpha) f(M^k\cdot  - \alpha) \Big\|_{C^\ell}=0.
\end{equation}
\end{definition}

\noindent For more details on test functions see \cite{DM97}. Note
that, it suffices to check \eqref{eq:C^l_convergence} for only one test
function $f$. Note also that, if all limits of a subdivision
scheme belong to $C^\ell(\RR^s)$, then the scheme may not converge
in $C^\ell$, but only in $C^0(\RR^s)$.

\smallskip \noindent
In this paper, we also show how to estimate the H\"older
regularity of subdivision limits.

\begin{definition} The \emph{H\"older regularity} of the $C^0-$convergent
scheme $\{S_{\ba^{(r)}}, \  r\in \NN\}$ is $\alpha=\ell+\zeta$, if
$\ell$ is the largest integer such that $g_{\bc} \in
C^\ell(\RR^s)$ and  $\zeta$ is the supremum of $\nu \in [0,1]$
such that
$$
 \max_{\mu \in \NN_0^s,  |\mu|=\ell} |D^\mu g_{\bc}(x)-D^\mu g_{\bc}(y)| \le
  |x-y|^\nu, \quad x,y \in \RR^s.
$$
We call $\alpha$ the \emph{H\"older exponent} of $\{S_{\ba^{(r)}},
\ r\in \NN\}$.
\end{definition}

\medskip \noindent
The joint spectral radius of a set of square matrices was
introduced in \cite{RotaStrang} and is independent of the choice
of the matrix norm $\|\cdot\|$.

\begin{definition}\label{def:JSR} The joint spectral radius (JSR) of a compact
family ${\cM}$ of square matrices  is defined by
$$
\displaystyle{ \rho({\cM}):=\lim_{n \rightarrow \infty}
\max_{M_{1}, \ldots, M_n \in \cM} \left\|\prod_{j=1}^n M_{j}
\right\|^{1/n}.}$$
\end{definition}

\noindent The link between the JSR and subdivision is well-known,
see e.g. \cite{Charina, CJR02, DL1992, J95, H03}.

\section{Parameter dependent subdivision schemes: matrix approach}\label{sec:parameters}

There are several examples of subdivision schemes that include a
tension parameter. We call them parameter dependent schemes. Often
the tension parameter is level dependent and shows a certain
asymptotic behavior which implies the asymptotic behavior of the
corresponding non-stationary scheme, i.e. $\displaystyle \lim_{r
\rightarrow \infty} \mathbf{a}^{(r)}=\mathbf{a}$. In this case,
although the set $\{A_{\varepsilon}^{(r)}, \ \varepsilon \in E, \
r \in \NN \}$ is not compact, the convergence and regularity of
the scheme $\{S_{\mathbf{a}^{(r)}}, \ r \in \NN\}$ can be analyzed
via the joint spectral radius approach in \cite{US2014}. The
results in \cite{US2014} are still applicable even if the
parameter values vary in some compact interval. Indeed, the
existence of the limit points for the sequence
$\{\mathbf{a}^{(r)}, \ r \in \NN\}$ of the subdivision masks  is
guaranteed, though these limit points are not always explicitly
known.

\begin{definition} \label{def:set_of_limit_points} For the mask sequence
$\{\mathbf{a}^{(r)},\ \ r \in \NN\}$ we denote by $\cA$ the set of
its limit points, i.e. the set of masks $\mathbf{a}$ such that
$$ \mathbf{a}\in \cA,\quad \hbox{if} \quad \exists \{r_n,\ n\in \NN \}\ \ \mbox{such that}\
 \ \lim_{n\rightarrow\infty}\mathbf{a}^{(r_n)}=\mathbf{a}\,.
$$
\end{definition}

In this section, we show that the joint spectral radius approach
can be effectively applied even if  the limit points of
$\{\mathbf{a}^{(r)}, \ r \in \NN\}$ cannot be determined
explicitly, but the masks $\mathbf{a}^{(r)}$ depend linearly on
the parameter $\omega^{(r)} \in [\omega_1, \omega_2]$,
$-\infty<\omega_1<\omega_2<\infty$.

Well-known and celebrated examples of parameter dependent stationary subdivision schemes with linear
dependence on the parameter are e.g. the univariate four point scheme \cite{DynGreLev87} with the
symbol
$$
a(z, \omega)=\frac{(1+z)^2}{2}+\omega(-z^{-2}+1+z^2-z^{4}), \quad
\omega \in \left[0,\frac{1}{16} \right], \quad z \in \CC \setminus
\{0\},
$$
which is a parameter perturbation of the linear B-spline. Also the bivariate butterfly scheme \cite{DLG90} with
the symbol
$$
 a(z_1,z_2, \omega)=\frac{1}{2}(1+z_1)(1+z_2)(1+z_1z_2)+\omega \,c(z_1,z_2), \quad z_1,z_2 \in \CC \setminus \{0\},
$$
with
\begin{eqnarray} \label{def:c_butterfly}
c(z_1,z_2)&=&z_1^{-1}z_2^{-2} + z_2^2 z_1^{-1} + z_1^{-2}z_2^{-1}
+ z_1^2z_2^{-1} - 2z_1^2z_2^3 - 2z_1^3z_2^2 + z_1^2z_2^4 +
z_1^4z_2^2 + z_1^3z_2^4 \notag \\&+& z_1^4z_2^3 - 2z_1^{-1} +
z_1^{-2} - 2z_1^2 - 2z_2^{-1} + z_1^3 + z_2^{-2} - 2z_2^2 + z_2^3
\end{eqnarray}
is a parameter perturbation of the linear three-directional box spline.
Other examples of such parameter dependent schemes are those with symbols that are convex combinations
$$
 \omega\, a(z_1,z_2)+(1-\omega)\,b(z_1,z_2)=b(z_1,z_2)+\omega\, (a(z_1,z_2)-b(z_1,z_2)), \quad \omega \in [0,1], \quad z_1,z_2 \in \CC \setminus \{0\},
$$
of two (or more) symbols of stationary schemes, see  e.g. \cite{GoPi2000, CoGoPiSa08, ContiCMS2010, CharinaContiJetterZimm11}. Known are also their non-stationary univariate counterparts with level dependent
parameters $\omega^{(r)}$ (see \cite{BCR2007_1, BCR2007_2, CoGoPi07, ContiRomani10}, for example)
$$
\frac{(1+z)^2}{2}+\omega^{(r)}(-z^{-2}+1+z^2-z^{4}), \quad r \in
\NN, \quad \lim_{r \rightarrow \infty}\omega^{(r)}=\omega \in \RR,
$$
and
$$
\omega^{(r)}\, a(z)+(1-\omega^{(r)})\,b(z), \quad r \in \NN, \quad
\omega^{(r)} \in [0,1].
$$
Note that the use of the level dependent parameters sometimes
allows us to enhance the properties of the existing stationary
schemes (e.g. with respect to their smoothness, size of their
support or reproduction and generation properties
\cite{CharinaContiJetterZimm11,ContiCMS2010,CoGoPi07,ContiRomani10}).

\smallskip
\noindent In all schemes considered above, the subdivision rules
depend either on the same, fixed, parameter $\omega=\omega^{(r)}
\in [\omega_1,\omega_2]$ independent of $r$, or the parameters
$\omega^{(r)} \in [\omega_1,\omega_2]$ are chosen in a such a way
that either $\displaystyle \lim_{r \rightarrow \infty}
\omega^{(r)}=\omega \in [\omega_1,\omega_2]$ or the corresponding
non-stationary scheme is asymptotically equivalent to some known
stationary scheme. In this section, we provide a matrix method for
analyzing regularity of more general subdivision schemes: we
consider the level dependent masks ${\mathbf a}(\omega^{(r)})=\{
\ra_{\alpha}(\omega^{(r)}), \ \alpha \in \ZZ^s\}$,  $r \in \NN$,
and require that $\omega^{(r)} \in [\omega_1,\omega_2]$ without
any further assumptions on the behavior of the sequence
$\{\omega^{(r)}, \ r \in \NN\}$. We assume, however, that each of
the masks depends linearly on the corresponding parameter
$\omega^{(r)}$.

The level dependent masks $\{{\mathbf a}(\omega^{(r)}), \ r \in \NN\}$ define the corresponding square matrices which we denote by
\begin{equation} \label{A_eps_r}
A_{\varepsilon, \omega^{(r)}}=\left(
\ra_{M\alpha+\varepsilon-\beta}(\omega^{(r)})\right)_{\alpha,\beta
\in K}, \quad \varepsilon \in E,
\end{equation}
and the level dependent symbols
$$
 a(z,\omega^{(r)})= \sum_{\alpha \in \ZZ^s} \ra_{\alpha}(\omega^{(r)}) z^{\alpha},
 \quad z^{\alpha}=z_1^{\alpha_1} \cdots z_s^{\alpha_s},\quad z\in \left(\CC \setminus \{0\} \right)^s.
$$

The assumption that each mask ${\mathbf a}(\omega^{(r)})$ depends linearly on $\omega^{(r)}$, leads to
the following immediate, but crucial result.

\begin{proposition} \label{prop:linearity} Let $\ell \in \NN_0$ and $-\infty<\omega_1<\omega_2 < \infty$.
If every symbol of the sequence $\{ a(z,\omega^{(r)}),\ r \in
\NN\}$ depends linearly on the parameter $\omega^{(r)} \in
[\omega_1,\omega_2]$ and satisfies sum rules of order $\ell+1$,
then every matrix in ${\cal T}=\{ A_{\varepsilon,
\omega^{(r)}}|_{V_\ell}, \  \omega^{(r)} \in [\omega_1,\omega_2],
\ \varepsilon \in E, \ r \in \NN\}$ is a convex combination of the
matrices with $\omega^{(r)} \in \{\omega_1, \omega_2\}$
$$
 A_{\varepsilon, \omega^{(r)}}|_{V_\ell}=(1-t^{(r)})A_{\varepsilon,\omega_1}|_{V_\ell}+t^{(r)}
 A_{\varepsilon, \omega_2}|_{V_\ell},\quad t^{(r)} \in [0,1].
$$
\end{proposition}

\begin{proof}
Let $r \in \NN$. We first write $\omega^{(r)}$ as a convex
combination of $\omega_1$ and $\omega_2$, i.e.
$$
 \omega^{(r)}=(1-t^{(r)}) \omega_1 +t^{(r)} \omega_2 \quad \hbox{with}\quad t^{(r)} \in [0,1]\,.
$$
Note that all entries of the matrices $A_{\varepsilon,
\omega^{(r)}}$, $\varepsilon \in E$, are the coefficients of the
corresponding mask  ${\mathbf a}(\omega^{(r)})$. Since the mask
coefficients depend linearly on the parameter $\omega^{(r)}$, so
do the matrices $A_{\varepsilon, \omega^{(r)}}$, and hence, the
corresponding linear operators. Therefore, the restrictions of
these operators to their common invariant subspace $V_\ell$ also
depend linearly on this parameter.
\end{proof}

\noindent In the level independent case, i.e.
$\omega^{(r)}=\omega$ for all $r \in \NN$, the use of the joint
spectral radius approach for studying the convergence and
regularity of the corresponding stationary subdivision schemes is
well understood. To show how this approach can be applied in the
our non-stationary setting, we need first to prove the following
auxiliary result.

\begin{proposition}\label{prop:Nicola} Let $\ell \in \NN_0$ and
\begin{equation}\label{def:A_omega}
{\cal T}= \{ A_{\varepsilon,\omega^{(r)}}|_{V_\ell},\ \omega^{(r)}
\in [\omega_1,\omega_2],\ \varepsilon \in E, \ r \in \NN\}
\end{equation}
be the infinite family of square matrices.  If the JSR of the
family ${\cal T}_{\omega_1, \omega_2}=\{
A_{\varepsilon,\omega_1}|_{V_\ell},
A_{\varepsilon,\omega_2}|_{V_\ell},  \ \varepsilon \in E \}$
satisfies $\rho({\cal T}_{\omega_1, \omega_2}) = \gamma,$ then
$\rho \left( {\cal T} \right) = \gamma$.
\end{proposition}

\begin{proof}
First of all observe that $\rho({\cal T}_{\omega_1, \omega_2}) = \gamma$ implies, for any $\delta > 0$,
the existence of a $\delta$-extremal norm (see e.g. \cite{Els95,GZ01}), i.e.  an operator
norm $\| \cdot \|_\delta$ such that
\begin{equation}
\| A_{\varepsilon,\omega_1}|_{V_\ell} \|_\delta \le \gamma +
\delta, \qquad \| A_{\varepsilon,\omega_2}|_{V_\ell} \|_\delta \le
\gamma + \delta. \label{eq:deltaextremality0}
\end{equation}
Then, by Proposition \ref{prop:linearity}, estimates in \eqref{eq:deltaextremality0} and
subadditivity of
matrix operator norms, we get
\[
\| A_{\varepsilon,\omega^{(r)}}|_{V_\ell} \|_\delta = \|
(1-t^{(r)}) A_{\varepsilon,\omega_1}|_{V_\ell} + t^{(r)}
A_{\varepsilon,\omega_2}|_{V_\ell}\|_\delta \le (1-t^{(r)}) \|
A_{\varepsilon,\omega_1}|_{V_\ell} \|_\delta + t^{(r)} \|
A_{\varepsilon,\omega_2}|_{V_\ell} \|_\delta = \gamma + \delta,
\quad t^{(r)} \in [0,1].
\]
This, due to the arbitrary choice of $\delta > 0$, implies that
$\rho \left( {\cal T} \right) = \gamma$, which concludes the proof.
\end{proof}

\begin{remark} \label{rem:JSRsubsets} $(i)$
Note that, if the family ${\cal T}_{\omega_1, \omega_2}$ is
non-defective, i.e. there exists an extremal norm $\| \cdot \|$
such that $ \max_{\varepsilon \in E} \left\{ \|
A_{\varepsilon,\omega_1}|_{V_\ell} \|, \ \|
A_{\varepsilon,\omega_2}|_{V_{\ell}} \| \right\} = \gamma, $ then
${\cal T}$ is also non-defective and all products of degree $d$ of
the associated product semigroup have maximal growth bounded by
$\gamma^d$. Note also that for any family of matrices ${\cal B}$,
${\cal B}\subset {\cal T}$, it  follows that
 $\rho \left( {\cal B}\right) \le \gamma$.
$(ii)$ Moreover, if a family $\mathcal T$ is irreducible, i.e.,
its matrices do not have a common nontrivial subspace, then
$\mathcal T$ is non-defective. Therefore, the case of
non-defective families is quite general.
\end{remark}

We are now ready to formulate the  main result of this section.

\begin{theorem}\label{teo:JSRregularity_r} Let $\ell \in \NN_0$. Assume that every
symbol of the sequence $\{a(z,\omega^{(r)}),\ r \in \NN\}$
depends linearly on $\omega^{(r)} \in [\omega_1,\omega_2]$ and
satisfies sum rules of order $\ell+1$. Then the non-stationary
scheme $\{S_{{\mathbf a}(\omega^{(r)})}, \ r \in \NN\}$ is
$C^\ell$-convergent, if  the JSR  of the family ${\cal
T}_{\omega_1, \omega_2} =\{ A_{\varepsilon,\omega_1}|_{V_{\ell}},
A_{\varepsilon,\omega_2}|_{V_{\ell}},  \ \varepsilon \in E \}$
satisfies
\begin{equation}
\rho({\cal T}_{\omega_1, \omega_2}) = \gamma < |m|^{-\ell}.
\label{eq:jsrA}
\end{equation}
Moreover the H\"{o}lder exponent of its limit functions is $\alpha \ge -\log_{|m|} \gamma$.
\end{theorem}

\begin{proof}
Since the  parameters $\{\omega^{(r)},\ r \in \NN\}$ vary in the
compact interval $[\omega_1, \omega_2]$, there exists a set of
limits points (finite or infinite) for the sequence $\{{\mathbf
a}(\omega^{(r)}),\ r \in \NN\}$ of subdivision masks. Let us
denote this set by $\cal A$ and the corresponding set of square
matrices by ${\cal
T}_\cA=\{A_\varepsilon=(\ra_{M\alpha+\varepsilon-\beta})_{\alpha,\beta
\in K}, \ \varepsilon \in E, \ {\mathbf a} \in \cA\}$. Obviously,
${\cal T}_\cA \subset {\cal T}$ with ${\cal T}$ as in
\eqref{def:A_omega}. Since by Proposition \ref{prop:Nicola} and
Remark \ref{rem:JSRsubsets}, $\rho \left( {\cal T}_\cA\right)\le
\gamma$, the claim follows by \cite[Corollary 1]{US2014}.
\end{proof}

\begin{remark} $(i)$ Note that, due to $\rho({\cal T}_\cA) \le \gamma$,
Theorem \ref{teo:JSRregularity_r}
yields a smaller H\"older exponent $\alpha$ than what could be
obtained by  \cite[Corollary 1]{US2014}. For example, consider
binary subdivision scheme with the symbols
\begin{eqnarray*}
a(z,\omega^{(r)})&=&z^{-1}\frac{(1+z)^2}{2}, \quad \quad r \in \{1,\ldots,L\}, \quad L \in \NN, \\
a(z,\omega^{(r)})&=&z^{-1}\frac{(1+z)^2}{2}+\frac{1}{16}(-z^{-3}+z^{-1}+z-z^{3}),\quad r\ge L+1,
\quad z \in \CC \setminus \{0\}.
\end{eqnarray*}
To apply Theorem \ref{teo:JSRregularity_r}, we can view the
corresponding masks as being linearly dependent on parameters
$\omega^{(r)} \in [0,\frac{1}{16}]$. The corresponding family
${\cal T}_{0,\frac{1}{16}}=\{ A_{\varepsilon,0}|_{V_1},
A_{\varepsilon,\frac{1}{16}}|_{V_1},  \ \varepsilon \in \{0,1\}
\}$ consists of the four  matrices
\begin{equation} \label{def:Matrices_4_point}
A_{0,\omega}|_{V_1}= \left( \begin{array}{rrrr}
      -\omega  & -2\omega+\frac{1}{2} &   -\omega &   0 \\
       0 & 2\omega &  2\omega &  0 \\
           0  &  -\omega &   -2\omega+\frac{1}{2} & -\omega \\
           0  & 0 &   2\omega &  2\omega
             \end{array} \right), \quad
   A_{1,\omega}|_{V_1}= \left( \begin{array}{rrrr}
     2\omega  &  2\omega &  0 &        0 \\
           -\omega  &  -2\omega+\frac{1}{2} &  -\omega &        0 \\
           0  & 2\omega &  2\omega &  0 \\
           0  & -\omega &  -2\omega+\frac{1}{2}   &  -\omega
             \end{array} \right)
\end{equation}
for $\omega \in\{0,\frac{1}{16}\}$. Due to
$$
 \max_{\varepsilon \in \{0,1\}} \left\{ \|A_{\varepsilon,0}|_{V_1}\|_\infty,
\|A_{\varepsilon,\frac{1}{16}}|_{V_1}\|_\infty \right\}= \max
_{\varepsilon \in \{0,1\}} \left\{ \rho(A_{\varepsilon,0}|_{V_1}),
\rho(A_{\varepsilon,\frac{1}{16}}|_{V_1})\right\}=\frac{1}{2},
$$
we get $\rho({\cal T}_{0, \frac{1}{16}}) =\frac12$ and, thus, the
corresponding scheme is convergent and has the H\"{o}lder exponent
$\alpha \ge 1$. On the other hand, the set $\cA$ of limit points
of the masks can be explicitly determined in this case and
consists of the mask of the four point scheme. Thus, by
\cite[Corollary 1]{US2014}, the H\"older exponent is actually
$\alpha \ge 2$.

$(ii)$ The regularity estimate given in  Theorem
\ref{teo:JSRregularity_r} can be improved, if the actual range of
the parameters  $\omega^{(r)}$, $r \ge L$, for some $L \in \NN$,
is a subinterval of  $[\omega_1,\omega_2]$, see section
\ref{subsec:examples}.

$(iii)$ Note that the result of Theorem \ref{teo:JSRregularity_r}
is directly extendable to the case when the matrix family $\cal T$
depends linearly on a convex polyhedral set
$\Omega=\overline{\hbox{co}\{\bomega_1, \ldots, \bomega_L \}}$ of
parameters $\bomega^{(r)} \in \Omega \subset \RR^p$, $r \in \NN$,
such that
\[
\bomega^{(r)} = \sum\limits_{j=1}^{L} t^{(r)}_j \bomega_j \quad
\mbox{with} \quad t^{(r)}_j \in [0,1] \quad  \mbox{and} \ \sum
\limits_{j=1}^{L} t^{(r)}_j=1.
\]
This is the case, for example, when we define the level and parameter dependent symbols
$$
a(z,\bomega^{(r)})=\sum \limits_{j=1}^{p} \omega_j^{(r)}a_j(z),
\quad \bomega^{(r)}=(\omega_1^{(r)}, \ldots, \omega_p^{(r)})^T \in
\Omega, \quad r \in \NN.
$$
\end{remark}

\subsection{Examples}\label{subsec:examples}

In this section we present two univariate examples of level dependent parameter
schemes, whose constructions are based on the four point and six point Dubuc-Deslauriers schemes.
In particular, in Example \ref{ex:4_point},
the non-stationary scheme is constructed in such a way that the support of its limit function
$$
 \phi_1=\lim_{r \rightarrow \infty} S_{{\mathbf a}(\omega^{(r)})} \ldots  S_{{\mathbf a}(\omega^{(1)})} \delta, \quad
\delta(\alpha)=\left\{ \begin{array}{cc} 1, & \alpha=0, \\ 0, & \hbox{otherwise} \end{array}\right., \quad
\alpha \in \ZZ^s,
$$
is smaller than the support of the four point scheme, but its regularity is comparable.
In Example \ref{ex:4_point_6_point}, every non-stationary mask is a convex combination of the four point and six point Dubuc-Deslauriers schemes. We show how the regularity of
the corresponding non-stationary scheme depends on the range of the corresponding parameters $\{\omega^{(r)}, \ r \in \NN\}$.
Both examples illustrate the importance of the
dependency on several parameters $\{\omega^{(r)}, \ r \in \NN\}$ instead of one $\omega \in \RR$.

\begin{example} \label{ex:4_point}
We consider the univariate, binary scheme with the symbols
\begin{eqnarray*}
a(z,\omega^{(r)})&=&z^{-1}\frac{(1+z)^2}{2}, \quad \quad r \in \{1,2\}, \\
a(z,\omega^{(r)})&=&z^{-1}\frac{(1+z)^2}{2}+\omega^{(r)}(-z^{-3}+z^{-1}+z-z^{3}),\quad r\ge 3,
\quad z \in \CC \setminus \{0\},
\end{eqnarray*}
where $\omega^{(r)}$ are chosen at random from the interval
$[\frac{3}{64},\frac{1}{16}]$. The corresponding family
$$
 {\cal T}_{0,\frac{1}{16}}=\{ A_{\varepsilon,0}|_{V_1},  A_{\varepsilon,\frac{1}{16}}|_{V_1},
\ \varepsilon \in \{0,1\} \}
$$
consists of the same four matrices as in \eqref{def:Matrices_4_point}.
And at the first glance the H\"older exponent of this scheme  is $\alpha \ge 1$.
On the other hand, we can view this scheme as the one with the corresponding matrix
family
$$
 {\cal T}_{\frac{3}{64},\frac{1}{16}}=\{ A_{\varepsilon,\frac{3}{64}}|_{V_1},  A_{\varepsilon,\frac{1}{16}}|_{V_1}, \
   \varepsilon \in \{0,1\} \},
$$
applied to a different starting data. Then we get $\rho({\cal T}_{\frac{3}{64},\frac{1}{16}})=3/8$ and,
by  Theorem \ref{teo:JSRregularity_r}, the H\"older exponent is actually $\alpha \ge
- \hbox{log}_2\frac{3}{8} \approx 1.4150$.

The size of the support of $\phi_1$ can be determined using the technique from  \cite{CohenDyn}
and is given by
$$
 \left[\sum_{k=0}^\infty 2^{-k-1} \ell(k), \sum_{k=0}^\infty 2^{-k-1} r(k) \right]=\left[-\frac{3}{2}, \frac{3}{2} \right]
$$
with
\begin{eqnarray*}
 \ell(k)&=&-1, \quad r(k)=1, \quad k=0,1, \\
 \ell(k)&=&-3, \quad r(k)=3, \quad k \ge 2.
\end{eqnarray*}
Recall that the support of the basic limit function of the four
point scheme is $[-3,3]$.

\end{example}

\medskip
\begin{example} \label{ex:4_point_6_point} In this example we consider the univariate non-stationary
scheme with symbols
$$
 a(z,\omega^{(r)})=\omega^{(r)} a(z) +(1-\omega^{(r)}) b(z), \quad
 \omega^{(r)} \in [0,1], \quad z \in \CC \setminus \{0\},
$$
where
$$
 a(z)=-\frac{z^{-3}(z+1)^4}{16}\left(z^2-4z+1\right)
$$
is the symbol of the four point scheme and
$$
 b(z)=\frac{z^{-5}(z+1)^6}{256}\left(3z^4-18z^3+38z^2-18z+3\right)
$$
is the symbol of $C^2-$convergent quintic Dubuc-Deslauriers subdivision
scheme \cite{DesDubuc}. By \cite{Floater}, the H\"older exponent of the $S_{\bf b}$ is $\alpha\approx 2.8301$. To determine the regularity of this
level and parameter dependent scheme we consider the matrix set
$$
  {\cal T}_{0,1}=\{ A_{\varepsilon,0}|_{V_2},  A_{\varepsilon,1}|_{V_2},
\ \varepsilon \in \{0,1\} \}
$$
with the four matrices
\begin{eqnarray*} \small
 A_{0,\omega}|_{V_2}&=&\frac{1}{256} \left(\begin{array}{rrrrrrr}
 3-3\omega& 0&   0& 0&   0&   0&    0\\
 -7-9\omega& -9+9\omega&  3-3\omega& 0& 0& 0& 0\\
 45+3\omega& 45+3\omega& -7-9\omega& -9+9\omega& 3-3\omega& 0& 0\\
 -9+9\omega& -7-9\omega& 45+3\omega& 45+3\omega& -7-9\omega& -9+9\omega& 3-3\omega\\
0& 3-3\omega& -9+9\omega& -7-9\omega& 45+3\omega& 45+3\omega& -7-9\omega\\
0& 0& 0&  3-3\omega& -9+9\omega& -7-9\omega& 45+3\omega\\
0& 0& 0& 0& 0& 3-3\omega&-9+9\omega
\end{array} \right), \\
 A_{1,\omega}|_{V_2}&=& \frac{1}{256} \left(\begin{array}{rrrrrrr}
  -9+9\omega&  3-3\omega&   0&            0& 0& 0& 0\\
 45+3\omega& -7-9\omega& -9+9\omega& 3-3\omega0& 0& 0\\
-7-9\omega& 45+3\omega& 45+3\omega& -7-9\omega& -9+9\omega &3-3\omega& 0\\
 3-3\omega& -9+9\omega& -7-9\omega& 45+3\omega& 45+3\omega& -7-9\omega&-9+9\omega\\
0& 0& 3-3\omega& -9+9\omega& -7-9\omega& 45+3\omega& 45+3\omega\\
0& 0& 0& 0& 3-3\omega& -9+9\omega& -7-9\omega\\
  0& 0& 0& 0& 0& 0&3-3\omega \end{array} \right)
\end{eqnarray*}
for $\omega \in \{0,1\}$. In this case, the regularity of the non-stationary  scheme $\{S_{{\bf a}(\omega^{(r)})}, \ r \in \NN\}$ coincides
with the regularity of the four point scheme. For $\omega \in \{a,1\}$, $a>0$, the scheme $\{S_{{\bf a}(\omega^{(r)})}, \ r \in \NN\}$
is $C^2-$convergent. And, for $\omega \in \{0,a\}$, $a <1$, extensive numerical experiments show that
the JSR of the family ${\cal T}_{0,a}$ is determined by the the subfamily $\{ A_{\varepsilon,a}|_{V_2}, \
\varepsilon \in \{0,1\} \}$. For example, for $a=\frac{1}{2}$, we obtain
$\rho( {\cal T}_{0,\frac{1}{2}}) \approx 0.2078$ and, thus, the
corresponding H\"older exponent is $\alpha \ge 2.2662$.
\normalsize
\end{example}

\section{Limitations of generation properties of non-stationary schemes}\label{sec:limitations}

It is known that certain level dependent (non-stationary) subdivision schemes are capable of
generating/reproducing certain spaces of exponential polynomials, see e.g.
\cite{CharinaContiRomani13, ContiRomani11}. In this section, we are interested in answering
the question: How big is the class of functions that can be generated/reproduced by such schemes?

\smallskip \noindent
More precisely, we show that, already in the univariate setting,
the zero sets of the Fourier transforms of the limit functions
$$
 \phi_k=\lim_{r \rightarrow \infty} S_{{\mathbf a}^{(r)}} \ldots  S_{{\mathbf a}^{(k)}} \delta, \quad
\delta(\alpha)=\left\{ \begin{array}{cc} 1, & \alpha=0, \\ 0, & \hbox{otherwise} \end{array}\right., \quad
\alpha \in \ZZ^s,
$$
of such schemes are unions of the sets
$$
 \Gamma_r=\{\omega \in \CC \ : \ a^{(r)}(e^{-i 2 \pi M^{-r}\omega})=0\}, \quad r \ge k,
$$
and that the sets $\Gamma_r$ are such that $\Gamma_r+ M^r \ZZ =\Gamma_r$.
Thus, some elementary functions cannot be generated by non-stationary schemes,
see Example \ref{ex:bad}.
Also the requirement that
$$
 \hat{\phi}_k(\omega)=\int_{\RR} \phi_k(x) e^{-i 2 \pi  x \omega} dx, \quad \omega \in \CC, \quad k \in \NN,
$$
is an entire function, limits the properties of the functions that can be generated by non-stationary subdivision schemes.

\begin{proposition}  \label{prop:limitations}
Let   $\{ \phi_k,\ k\in \NN\}$  be continuous functions of compact support  satisfying
$$
 \phi_k(x)=\sum_{\alpha \in \ZZ} \ra^{(k)}(\alpha) \phi_{k+1}(Mx-\alpha), \quad k \in
\NN, \quad x \in \RR.
$$
Then
$$
 \{\omega \in \CC \ : \hat{\phi_k}(\omega)=0\}=\bigcup_{r \ge k} \Gamma_r,
$$
such that the sets $\Gamma_r$  satisfy
$$
 \Gamma_r+  M^r \ZZ =\Gamma_r.
 $$
\end{proposition}
\begin{proof} Let $k \in \NN$. By Paley-Wiener theorem, the Fourier transform $\hat{\phi}_k$ defined
on $\RR$ has an analytic extension
$$
 \hat{\phi}_k(\omega)=\int_{\RR} \phi_k(x) e^{-i 2 \pi  x \omega} dx, \quad \omega \in \CC,
$$
to the whole complex plane $\CC$ and $ \hat{\phi}_k$ is an entire function. By
Weierstrass theorem \cite{Conway}, every entire function can be represented by a
product involving its zeroes. Define the sets
$$
 \Gamma_r:=\{\omega \in \CC \ : \ a^{(r)}(e^{-i 2 \pi M^{-r}\omega})=0\}, \quad r \in \NN.
$$
Let $z_{r,1}, \ldots, z_{r,N}$ be the zeros of the polynomials
$a^{(r)}(e^{-i 2 \pi M^{-r}\omega})$, counting their
multiplicities. Then
$$
 \Gamma_r= i M^r \bigcup_{\ell=1}^N \hbox{Ln}(z_{r,\ell}),
$$
where, by the properties of the complex logarithm, each of the sets $i M^r
\hbox{Ln}(z_{r,\ell})$  consists of sequences  of complex numbers
and is $ M^r -$periodic. Thus, each of the sets $\Gamma_r$ satisfy
$$
 \Gamma_r+M^r \ZZ=\Gamma_r, \quad r \in \NN.
$$
The definition of $\hat{\phi}_k$  as an infinite product of the
trigonometric polynomials $a^{(r)}(e^{-i 2 \pi M^{-r}\omega})$, $r
\ge k$, yields the claim.

\end{proof}
\smallskip

The following  examples  illustrate the result of Proposition \ref{prop:limitations}.

\begin{example}
The basic limit function of the simplest stationary scheme is given by
$\phi_1=\chi_{[0,1)}$. Its Fourier transform
is
$$
 \hat{\phi}_1(\omega)=\frac{1-e^{-i2\pi \omega}}{i 2 \pi  \omega}, \quad \hbox{and} \quad \{\omega
\in \CC \ : \ \hat{\phi}_1(\omega)=0\}=  \ZZ \setminus \{0\}.
$$
The mask symbol $a(z)=1+z$ has a single zero at $z=-1$, i.e.
$e^{-i2 \pi 2^{-r}\omega}=-1$ for $\omega = 2^r \{ \frac{1}{2}+  k
\: \ k\in \ZZ  \}$, $r \in \NN_0$. In other words, $\Gamma_1=\{1+
2 k \ : \ k\in \ZZ  \}$ and $\Gamma_r=2 \Gamma_{r-1}$ for  $r \ge
2$. Therefore,
$$
  \{\omega \in \CC \ : \ \hat{\phi}_1(\omega)=0\}=\bigcup_{r \in \NN} \Gamma_r.
$$
\end{example}

\begin{example}  The first basic limit function of the simplest non-stationary scheme is
given by $\phi_1(x)=\chi_{[0,1)}(x) e^{\lambda x}$,
$\lambda \in \CC$. Its Fourier transform
is
$$
 \hat{\phi}_1(\omega)=\frac{e^{-i 2\pi \omega+\lambda}-1}{-i2 \pi \omega+\lambda}, \quad \omega
\in \CC, \quad \hbox{and} \quad \{\omega \in \CC \ : \ \hat{\phi}_1(\omega)=0\}=-\frac{i
\lambda}{2\pi} +\ZZ \setminus\{0\} .
$$
The mask symbol $a^{(r)}(z)=1+e^{\lambda 2^{-r}}z$ has a single
zero at $z=-e^{-\lambda 2^{-r}}$, i.e. $e^{-i2 \pi
2^{-r}\omega}=-e^{-\lambda 2^{-r}}$ for $\omega =
 -\frac{i \lambda}{2\pi}+ 2^r \{ \frac{1}{2}+ k \ : \ k\in \ZZ  \}$, $r \in \NN$.
Note that $\Gamma_1= -\frac{i \lambda}{2\pi}+\{1+2 k \ : \ k\in \ZZ  \}$ and
$$
 \bigcup_{r \in \NN} 2^r\{\frac{1}{2}+ k \ : \ k\in \ZZ  \}=\ZZ \setminus\{0\}.
$$
Therefore,
$$
  \{\omega \in \CC \ : \ \hat{\phi}_1(\omega)=0\}=\bigcup_{r \in \NN} \Gamma_r.
$$
\end{example}

In the next example we identify a compactly supported function
that cannot be generated by any non-stationary subdivision scheme.

\begin{example} \label{ex:bad}
Let us consider the compactly supported function
$$
 f(x)=\chi_{[-1,1]}(x) \frac{2}{\sqrt{1-x^2}}, \quad x \in \RR.
$$
It cannot be a limit of any non-stationary  subdivision scheme. Indeed, its Fourier transform
\begin{equation} \label{eq:FT_alternative}
  J_0(\omega)=\int_{\RR} f(x) e^{-i  x \omega} dx, \quad \omega \in \CC,
\end{equation}
is the Bessel function $J_0$ of the first kind, which is entire, but has only positive
zeros.
The lower bound for its zeros $j_{0,s}$, $s \in \NN$, is given by
$j_{0,s} > \sqrt{(s-\frac{1}{4})^2 \pi ^2}$, see \cite{McCann}. Thus, Proposition
\ref{prop:limitations}
implies the claim.
\end{example}

{\bf Acknowledgements:} Vladimir Protasov was sponsored by RFBR grants $13-01-00642$, $14-01-00332$
and by the grant of Dynasty foundation.


\end{document}